\documentclass[12pt]{amsart}
\usepackage{bm,amsmath,amsthm,amssymb,mathtools,verbatim,amsfonts,tikz-cd,mathrsfs,mathabx}
\usepackage{enumitem}
\usepackage[hidelinks,colorlinks=true,linkcolor=blue, citecolor=black,linktocpage=true]{hyperref}
\usepackage[alphabetic]{amsrefs}

\usepackage[top=1.1in, bottom=1.1in, left=1.3in, right=1.3in,marginpar=1.1in]{geometry}

\newtheorem*{theorem*}{Theorem}
\newtheorem*{question*}{Question}
\theoremstyle{remark}
\newtheorem*{remark*}{Remark}

\title{The topological and smooth Hausmann--Weinberger invariants disagree}

\author{Mike Miller Eismeier}

\address{University of Vermont\\ USA}
\email{Mike.Miller-Eismeier@uvm.edu}
\begin{document}
\maketitle

\begin{abstract}\vspace{-1cm}
For $\pi$ a finitely presented group, Hausmann and Weinberger defined $q(\pi) \in \mathbb Z$ to be the minimum Euler characteristic over all closed, oriented $4$-manifolds with fundamental group $\pi$. This short note establishes that this minimum value in general differs depending on whether one minimizes over topological manifolds or only those admitting a smooth structure.\\ 
\end{abstract}

It is an important fact in 4-dimensional topology that any finitely presented group arises as the fundamental group of an explicit closed oriented $4$-manifold; this fact is often used to justify a focus on simply connected manifolds. Nevertheless, many interesting problems remain in the general case. To start, which finitely presented groups are the fundamental groups of manifolds with some additional geometric structure? What is the landscape of $4$-manifolds whose fundamental group is fixed?

Towards this second question, Hausmann and Weinberger \cite{HW} define the integer $q^{\mathsf{DIFF}}(\pi)$ as the minimal Euler characteristic over all closed \textit{smooth} oriented $4$-manifolds with fundamental group $\pi$. It is notoriously difficult to compute this quantity; its determination is posed as Problem 4.59 of the 1995 Kirby Problem List, which is open outside a select few cases. Kotschick computes $q$ for cyclic groups, surface groups, products of two surface groups, and computes the behavior under free product with $\mathbb Z$. Kirk and Livingston \cite{KL} compute $q$ for free abelian groups and certain direct products of cyclic groups. Hillman \cite{Hillman} computes $q$ for all groups of cohomological dimension $2$. Hildum \cite{Hildum} computes $q$ for a class of right-angled Artin groups and poses a conjecture for general RAAGs. Adem and Hambleton \cite{AH}, among other things, show that $q(A_4) = q(A_5) = 4$. Sun and Wang \cite{SW} compute $q$ for many $3$-manifold groups, extending other work of Hillman \cite{Hillman0} and Kirk--Livingston \cite{KL2}.

Because the smooth and topological categories are quite different in dimension $4$, it is also interesting to study the corresponding quantity $q^{\mathsf{TOP}}(\pi)$ where the minimum is instead taken over closed, oriented, \textit{topological} $4$-manifolds with fundamental group $\pi$. In all the examples enumerated above, it happens that $q^{\mathsf{TOP}}(\pi) = q^{\mathsf{DIFF}}(\pi)$. Kotschick asks in \cite[Problem 2.7]{Kotschick} whether this equality holds for all $\pi$. This question was repeated in \cite{KL, KL2}. Its answer is negative: 

\begin{theorem*}\label{thm:main}
There exists a finitely-presented group $\pi$ for which $q^{\text{\textup{\textsf{TOP}}}}(\pi)< q^{\text{\textup{\textsf{DIFF}}}}(\pi)$.
\end{theorem*}

\begin{proof}
One such group $\pi$ is the fundamental group of the manifold $N$ of \cite[Theorem 5a.1]{DJ}, which is a closed, oriented, aspherical, topological $4$-manifold which has $w_2(N) = 0$ and $\sigma(N) = 8$. Because $N$ is aspherical, $\chi(N) = q^{\mathsf{TOP}}(\pi)$ (see e.g. \cite[Theorem 3.8]{Kotschick} or the argument below). Rokhlin's theorem \cite{Rokhlin} implies that $N$ is not smoothable. 

Now suppose $M$ is any closed, oriented, topological $4$-manifold with $\pi_1(M) \cong \pi$ and $\chi(M) = \chi(N)$. Consider the natural map $f: M \to N = K(\pi, 1)$ induced by the given isomorphism. $f$ induces an isomorphism on $H^0$ and $H^1$ with any coefficients, and an injection on $H^2$ by Hopf's exact sequence. Let $\mathbb F$ be a field. Because \[\chi(M) = 2b_0(M;\mathbb F) - 2b_1(M;\mathbb F) + b_2(M;\mathbb F)\] and similarly for $\chi(N)$, it follows that $b_2(M;\mathbb F) = b_2(N; \mathbb F)$. Therefore $f$ induces an isomorphism on $H^2$ with $\mathbb F$ coefficients, and in particular preserves the intersection form over $\mathbb F$. Taking $\mathbb F = \mathbb F_2$ shows that $w_2(M) = 0$; taking $\mathbb F = \mathbb R$ shows that $\sigma(M) = 8$. Thus $M$ is also non-smoothable, and therefore $q^{\mathsf{TOP}}(\pi) < q^{\mathsf{DIFF}}(\pi)$.
\end{proof}

It was pointed out to the author by Jonathan Hillman that the map $f$ above is in fact a homotopy equivalence, and indeed this is the case for the natural map $f: M \to N$ whenever $N$ is a closed, oriented, aspherical $PD_4$-complex for which $\chi(M) = \chi(N)$ and $\pi_1(M) \cong \pi_1(N)$ \cite[Corollary 3.5.1]{Hillman0}. One should not be too optimistic: it is not, in general, the case that a $\chi$-minimizer for $\pi$ (in any category) is unique up to homotopy equivalence. Indeed, for $\pi = \mathbb Z/2$, there are two smooth $S^2$-bundles over $\mathbb{RP}^2$ with orientable total space. Because Euler characteristic is multiplicative and $2b_0(\pi) - 2b_1(\pi) + b_2(\pi) = 2$, these are both $\chi$-minimizers. They are not homotopy equivalent: one is spin, and the other is not \cite[Remark 4.5]{HK}.\\

Kotschick defines another invariant $p(\pi)$ by minimizing $\chi - \sigma$ over all closed, oriented $4$-manifolds with fundamental group $\pi$. One might also study the minimum $p_{\text{spin}}(\pi)$ of $\chi - \sigma$ over closed \textit{spin} manifolds. Taking the connected sum with $\mathbb{CP}^2$ or Freedman's $E_8$ manifold leaves $\chi - \sigma$ unchanged but invalidates the argument above, so it remains unclear and interesting whether or not $p^{\mathsf{DIFF}}(\pi) = p^{\mathsf{TOP}}(\pi)$; similarly with its spin variation. 

If the Davis--Januzkiewicz manifold were positive-definite, Donaldson's diagonalization theorem would imply that $p^{\mathsf{TOP}}(\pi) < p^{\mathsf{DIFF}}(\pi)$. However, the hyperbolization procedure greatly increases the rank of homology. It seems natural to ask whether there is an aspherical $E_8$ manifold. More generally:

\begin{question*}
Suppose $H_2$ is a unimodular symmetric bilinear form over $\mathbb Z$. Does there exist a closed, oriented, \textbf{aspherical} topological $4$-manifold with intersection form $H_2$? Can one further assume that the first homology is trivial?
\end{question*}

The case $H_1 = H_2 = 0$ was only resolved fairly recently \cite{RT}. In addition, few example are known with odd Euler characteristic \cite{Edmonds}. It is unlikely that the cohomology ring can be prescribed arbitrarily, as it seems plausible that a closed, oriented, aspherical $4$-manifold satisfies $\chi(M) \ge |\sigma(M)|$ \cite{Gromov,Lueck}; for more details see the introduction of \cite{ACS}.\\

\noindent\textbf{Acknowledgements.} I am grateful for the hospitality of the Institute of Mathematics at Academia Sinica, where this note was written. I am grateful to Ian Hambleton, Jonathan Hillman, Danny Ruberman, and Balarka Sen for helpful comments. In particular, the example with $\pi = \mathbb Z/2$ was suggested by Hambleton, and the observation that $f$ is a homotopy equivalence was made by Hillman.



\bibliography{biblio.bib}
\bibliographystyle{alpha}
\end{document}